\documentclass[10pt]{amsart}
\usepackage[leqno]{amsmath}
\usepackage{amssymb,latexsym,soul,cite,amsthm,color,enumitem,graphicx,tikz,mathtools,microtype,accents}
\usepackage[colorlinks=true,urlcolor=blue(ncs),citecolor=blue(ncs),linkcolor=blue(ncs),linktocpage,pdfpagelabels,bookmarksnumbered,bookmarksopen]{hyperref}
\definecolor{blue(ncs)}{rgb}{0.0, 0.53, 0.74}
\usepackage[english]{babel}
\usepackage[left=2.5cm,right=2.5cm,top=2.5cm,bottom=2.5cm]{geometry}

\usepackage{enumitem}
\newcommand{\subscript}[2]{$#1 _ #2$}

\numberwithin{equation}{section}

\newtheorem{theorem}{Theorem}[section]
\theoremstyle{plain}
\newtheorem{lemma}[theorem]{Lemma}
\theoremstyle{plain}
\newtheorem{proposition}[theorem]{Proposition}
\theoremstyle{plain}
\newtheorem{corollary}[theorem]{Corollary}
\newtheorem{theo}{Theorem}

\theoremstyle{definition}
\newtheorem{remark}[theorem]{Remark}
\newtheorem{example}[theorem]{Example}
\newcommand{\N}{{\mathbb N}}

\newcommand{\R}{{\mathbb R}}
\newcommand{\eps}{\varepsilon}
\newcommand{\beq}{\begin{equation}}
\newcommand{\eeq}{\end{equation}}
\renewcommand{\le}{\leqslant}
\renewcommand{\ge}{\geqslant}

\newcommand{\h}{H^{s}_0(\Omega)}
\newcommand{\fl}{(-\Delta)^s\,}
\newcommand{\ds}{{\rm d}_\Omega^s}

\makeatletter
\newcommand{\leqnomode}{\tagsleft@true}
\newcommand{\reqnomode}{\tagsleft@false}
\makeatother

\newenvironment{enumroman}{\begin{enumerate}

}{\end{enumerate}}

\title[Positive solutions for the fractional Laplacian]{Existence and multiplicity of positive solutions for the fractional Laplacian under subcritical or critical growth}

\author[S.\ Frassu, A.\ Iannizzotto]{Silvia Frassu, Antonio Iannizzotto}

\address[S.\ Frassu, A.\ Iannizzotto]{Department of Mathematics and Computer Science
\newline\indent
University of Cagliari
\newline\indent
Via Ospedale 72, 09124 Cagliari, Italy}
\email{silvia.frassu@unica.it, antonio.iannizzotto@unica.it} 


\begin{document}

\begin{abstract}
We study a Dirichlet type problem for an equation involving the fractional Laplacian and a reaction term subject to either subcritical or critical growth conditions, depending on a positive parameter. Applying a critical point result of Bonanno, we prove existence of one or two positive solutions as soon as the parameter lies under a (explicitly determined) threshold. As an application, we find two positive solutions for a fractional Brezis-Nirenberg problem.
\end{abstract}

\maketitle

\begin{center}
Version of \today\
\end{center}

\section{Introduction}\label{sec1}

\noindent
This paper is devoted to the following Dirichlet problem for a pseudo-differential equation of fractional order:
\beq\label{p}
\begin{cases}
\fl u = \lambda f(u)& \text{in $\Omega$} \\
u>0 & \text{in $\Omega$} \\
u=0 & \text{in $\Omega^c$.}
\end{cases}
\eeq
Here $s\in(0,1)$, $\Omega\subset\R^N$ ($N>2s$) is a bounded domain with $C^{1,1}$ boundary, and the leading operator is the fractional Laplacian defined for all $u \in \mathcal{S}(\R^N)$ by
\beq\label{fl}
\fl u(x)= 2 \, \mathrm{P.V.} \int_{\R^N}\frac{u(x)-u(y)}{|x-y|^{N+2s}}\,dy.
\eeq
The autonomous reaction $f \in C(\R)$ is assumed to be non-negative and dominated at infinity by a power of $u$, namely, for all $t\in\R$
\beq\label{gr}
0\le f(t)\le a_0(1+|t|^{q-1}) \ (a_0>0,\,q\le 2^*_s),
\eeq
where $2^*_s = 2N/(N-2s)$ denotes the critical exponent for the fractional Sobolev space $H^s(\R^N)$ (see \cite{DPV}). Finally, $\lambda>0$ is a parameter.
\vskip2pt
\noindent
Problem \eqref{p} admits a variational formulation by means of the energy functional
\[J_\lambda(u) = \frac{[u]_s^2}{2}-\lambda\int_\Omega F(u)\,dx,\]
where $[\,\cdot\,]_s$ denotes the Gagliardo seminorm and $F$ is the primitive of $f$, i.e., weak solutions of \eqref{p} coincide with critical points of $J_\lambda$ in a convenient subspace of $H^s(\R^N)$ (see Section \ref{sec2} below for details). We note that, for $\lambda=1$, problem \eqref{p} embraces the following Dirichlet problem with pure power nonlinearities:
\beq \label{pp}
\begin{cases}
\fl u = \mu u^{p-1} + u^{q-1} & \text{in $\Omega$} \\
u > 0 & \text{in $\Omega$} \\
u=0 & \text{in $\Omega^c$,}
\end{cases}
\eeq
with $1<p<q\le 2^*_s$ and $\mu>0$.
\vskip2pt
\noindent
For a general introduction to the fractional Laplacian we refer to \cite{BV,CS,CS1,DPV}. The study of \eqref{p} (or closely related problems) via variational methods started from the work of Servadei and Valdinoci \cite{SV,SV1}. Here we distinguish between the {\em subcritical} ($q<2^*_s$ in \eqref{gr}) and {\em critical} ($q=2^*_s$) cases. In the subcritical case, we mention for instance the contributions of \cite{BMS,CW,DI,F,F1,IP,MP,WS} and the monograph \cite{MRS}.
\vskip2pt
\noindent
In the critical case, the main difficulty lies in the fact that $J_\lambda$ does not satisfy the (usual in variational methods) Palais-Smale compactness condition. In particular, problem \eqref{pp} with $q=2^*_s$ represents a fractional counterpart of the famous Brezis-Nirenberg problem \cite{BN}. Again, the first result in this direction is due to Servadei and Valdinoci \cite{SV2} (see also \cite{AP,BCDS,MS}). Later, Barrios {\em et al.} \cite{BCSS} studied \eqref{pp} with $1<p<2<q=2^*_s$ (concave case) proving that, for $\mu>0$ small enough, such problem has at least two positive solutions $u_\mu<w_\mu$, employing both topological (sub-supersolutions) and variational methods, in a way that was first introduced in \cite{ABC} for elliptic PDE's.
\vskip2pt
\noindent
Our approach to problem \eqref{p} is purely variational, mainly based on a critical point theorem of Bonanno \cite{B} and some of its consequences, presented in \cite{BD,BD1,BDO}. The main feature of such method is a strategy to find a local minimizer of a $J_\lambda$-type functional, which only requires a {\em local} Palais-Smale condition. Our results are the following:
\begin{itemize}[leftmargin=1cm]
\item[$(a)$] In the subcritical case ($q<2^*_s$) we apply an abstract result of \cite{BD} and explicitly compute a real number $\lambda^*>0$ s.t. problem \eqref{p} admits at least two positive solutions $u_\lambda$, $v_\lambda$ for all $\lambda\in(0,\lambda^*)$.
\item[$(b)$] In the critical case ($q=2^*_s$) we first study a generalization of problem \eqref{pp}, explicitly determining a real number $\mu^*>0$ s.t.\ there exist at least one positive solution $u_\mu$ for all $\mu\in(0,\mu^*)$. Then, we focus on \eqref{pp} with $q=2^*_s$ and, applying the mountain pass theorem, we produce a second positive solution $w_\mu>u_\mu$ for all $\mu\in(0,\mu^*)$ (here we mainly follow \cite{BDO}).
\end{itemize}
\vskip2pt
\noindent
To our knowledge, this is the first application of the ideas of \cite{B} in the field of fractional Laplacian equations. A noteworthy difference with respect to the classical elliptic case is the following: in this approach, it is essential to explicitly compute $J_\lambda(\bar u)$ at some Sobolev-type function $\bar u:\Omega\to\R$, which is usually chosen in such a way to have a piecewise constant $|\nabla\bar u|$. In the fractional framework, functions may have no gradient at all, and the computation of the Gagliardo seminorm is often prohibitive, so $\bar u$ will be chosen as (a multiple of) the solution of a fractional torsion equation in a ball (see \eqref{T}). 
\vskip2pt
\noindent
We also remark that our main result in part $(b)$ is formally equivalent to the main result of \cite{BCSS}, but with two substantial differences: the first solution $u_\mu$ is found as a local minimizer of $J_\lambda$ (instead of being detected via sub-supersolutions, and {\em a posteriori} proved to be a minimizer), and moreover the interval $(0,\mu^*)$ is {\em explicitly} determined (though possibly not optimal).
\vskip2pt
\noindent
The paper has the following structure: in Section \ref{sec2} we collect the necessary preliminaries; in Section \ref{sec3} we develop part $(a)$ of our study; in Sections \ref{sec4} and \ref{sec5} we focus on part $(b)$.
\vskip4pt
\noindent
{\bf Notation:} Throughout the paper, for any $A\subset\R^N$ we shall set $A^c=\R^N\setminus A$. By $|A|$ we will denote either the $N$-dimensional Lebesgue measure or the $(N-1)$-dimensional Hausdorff measure of $A$, which will be clear from the context. For any two measurable functions $u$, $v$, $u=v$ in $A$ will stand for $u(x)=v(x)$ for a.e.\ $x\in A$ (and similar expressions). We will often write $t^\nu=|t|^{\nu-1}t$ for $t\in\R$, $\nu>1$. For any $t\in\R$ we set $t^\pm=\max\{\pm t,0\}$. By $B_r(x)$ we denote the open ball centered at $x\in\R^N$ of radius $r>0$. For all $\nu\in[1,\infty]$, $\|\cdot\|_\nu$ denotes the standard norm of $L^\nu(\Omega)$ (or $L^\nu(\R^N)$, which will be clear from the context). Every function $u$ defined in $\Omega$ will be identified with its $0$-extension to $\R^N$. Moreover, $C$ will denote a positive constant (whose value may change line by line).

\section{Preliminaries}\label{sec2}

\noindent
We begin by recalling some basic notions about fractional Sobolev spaces (for details we refer to \cite{DPV}). We define the Gagliardo seminorm by setting for all measurable $u:\R^N \to \R$
\[[u]_{s} = \Big[ \iint_{\R^N\times\R^N}\frac{(u(x)-u(y))^2}{|x-y|^{N+2s}}\,dx\,dy \Big]^{\frac{1}{2}}.\]
Accordingly, we define the space
\[H^{s}(\R^N) = \big\{u\in L^2(\R^N):\,[u]_{s}<\infty\big\}.\]
The embedding $H^s(\R^N)\hookrightarrow L^{2^*_s}(\R^N)$ is continuous, and the fractional Talenti constant is given by the following lemma (see \cite[Theorem 1.1]{CT} and \cite[Proposition 3.6]{DPV}):

\begin{lemma} \label{Opt}
We have
\[T(N,s) = \max_{u \in H^s(\R^N)\setminus \{0\}} \frac{\|u\|_{2_s^*}}{[u]_s} = \frac{s^{\frac{1}{2}} \Gamma(\frac{N-2s}{2})^{\frac{1}{2}} \Gamma(N)^{\frac{s}{N}}} {2^{\frac{1}{2}} \pi^{\frac{N+2s}{4}} \Gamma(1-s)^{\frac{1}{2}} \Gamma(\frac{N}{2})^{\frac{s}{N}}} >0,\]
the maximum being attained at the functions
\[u(x)= \frac{a}{(b+|x-x_0|^2)^{\frac{N-2s}{2}}} \quad (a,b>0,\,x_0\in \R^N).\]
\end{lemma}

\noindent
Now we establish a variational formulation for \eqref{p}, following \cite{SV1} (see also \cite{IMS}). Set 
\[\h=\big\{u \in H^s(\R^N): u=0 \text{ in } \Omega^c\big\},\]
a Hilbert space under the inner product
\[\langle u,v\rangle = \iint_{\R^N\times\R^N}\frac{(u(x)-u(y))(v(x)-v(y))}{|x-y|^{N+2s}}\,dx\,dy\]
and the corresponding norm $\|u\|=[u]_s$ (see \cite[Lemma 7]{SV1}). The dual space of $\h$ is denoted $H^{-s}(\Omega)$. By Lemma \ref{Opt} and H\"{o}lder's inequality, for any $\nu\in [1, 2_s^*]$ the embedding $\h \hookrightarrow L^{\nu}(\Omega)$ is continuous and for all $u\in\h$ we have
\beq \label{Ho}
\|u\|_{\nu} \le T(N,s)|\Omega|^{\frac{2_s^*-\nu}{2_s^* \nu}}\|u\|.
\eeq
Further, the embedding is compact iff $\nu < 2_s^*$ (see \cite[Lemma 8]{SV1}). 
\vskip2pt
\noindent
In order to deal with problem \eqref{p} variationally, we assume the following hypotheses on the reaction $f$:

\begin{itemize}[leftmargin=1cm]
\item[${\bf H}_0$] $f \in C(\R)$, $F(t)=\int_0^t f(\tau)\,d\tau$, and
\begin{enumroman}
\item\label{h01} $f(t) \ge 0$ for all $t \in \R$;
\item\label{h02} $f(t) \le a_0(1+|t|^{2_s^*-1})$ for all $t \in \R$ ($a_0>0$).
\end{enumroman}
\end{itemize}

\noindent
We set for all $u \in \h, \; \lambda >0$
\[\Phi(u)= \frac{\|u\|^2}{2}, \quad \Psi(u)=\int_{\Omega} F(u)\,dx, \quad J_{\lambda}(u)=\Phi(u)- \lambda \Psi(u)\]
($\Psi$ is well defined by virtue of hypothesis ${\bf H}_0$ \ref{h01} \ref{h02}). Then $\Phi, \Psi, J_\lambda\in C^1(\h)$ with
\[\langle J_\lambda^{'}(u),\varphi\rangle = \langle u,\varphi\rangle-\lambda\int_\Omega f(u)\varphi\,dx\]
for all $u,\varphi\in\h$. We say that $u$ is a (weak) solution of problem \eqref{p} if $J'_\lambda(u)=0$ in $H^{-s}(\Omega)$, that is, for all $\varphi\in\h$ we have
\beq\label{ws}
\langle u,\varphi\rangle=\lambda\int_\Omega f(u)\varphi\,dx.
\eeq
The regularity theory for fractional Dirichlet problems was essentially developed in \cite{RS} (see also \cite{IMS,BCSS}). While smooth in $\Omega$, solutions are in general singular on $\partial\Omega$, so the best global regularity we can expect is weighted H\"older continuity, in the following sense. Set for all $x\in\overline\Omega$
\[{\rm d}_\Omega(x) = {\rm dist}(x,\Omega^c),\]
then define the spaces
\[C^0_s(\overline\Omega) = \Big\{u\in C^0(\overline\Omega):\,\frac{u}{\ds}\in C^0(\overline\Omega)\Big\}, \quad \|u\|_{0,s} = \Big\|\frac{u}{\ds}\Big\|_\infty,\]
and for any $\alpha\in(0,1)$
\[C^\alpha_s(\overline\Omega) = \Big\{u\in C^0(\overline\Omega):\,\frac{u}{\ds}\in C^\alpha(\overline\Omega)\Big\}, \quad \|u\|_{\alpha,s} = \|u\|_{0,s}+\sup_{x\neq y}\frac{|u(x)/\ds(x)-u(y)/\ds(y)|}{|x-y|^\alpha}.\]
The positive order cone of $C^0_s(\overline\Omega)$ has a nonempty interior given by
\[{\rm int}(C^0_s(\overline\Omega)_+) = \Big\{u\in C^0_s(\overline\Omega):\,\frac{u}{\ds}>0 \ \text{in $\overline\Omega$}\Big\}.\]
For the reader's convenience we recall from \cite[Theorems 2.3, 3.2 and Lemma 2.7]{IMS} the main properties of weak solutions:

\begin{proposition}\label{reg}
Let ${\bf H}_0$ hold, $u\in\h$ be a weak solution of \eqref{p}. Then:
\begin{enumroman}
\item\label{reg1} (a priori bound) $u\in L^\infty(\Omega)$;
\item\label{reg2} (regularity) $u\in C^\alpha_s(\overline\Omega)$ with $\alpha\in(0,s]$ depending only on $s$ and $\Omega$;
\item\label{reg3} (Hopf's lemma) if $u\neq 0$, then $u\in{\rm int}(C^0_s(\overline\Omega)_+)$.
\end{enumroman}
\end{proposition}

\noindent
By Proposition \ref{reg} \ref{reg3} we see that, whenever $u\in\h\setminus\{0\}$ satisfies \eqref{ws}, then in particular $u>0$ in $\Omega$. Moreover, assuming further that $f$ is locally Lipschitz in $\R$, from \cite[Corollary 1.6]{RS} we deduce that $u\in C^\beta(\Omega)$ for any $\beta\in[1,1+2s)$, which along with Proposition \ref{reg} \ref{reg2} implies that for all $x\in\R^N$ the mapping
\[x\mapsto\frac{u(x)-u(y)}{|x-y|^{N+2s}}\]
lies in $L^1(\R^N)$. Then, testing \eqref{ws} with any $\varphi\in C^\infty_c(\Omega)$ and applying \eqref{fl}, we have
\[\int_\Omega\fl u\varphi\,dx = \langle u,\varphi\rangle = \int_\Omega f(u)\varphi\,dx,\]
i.e., $u$ solves \eqref{p} pointwisely.
\vskip2pt
\noindent
We also recall the following result, relating the local minimizers of the energy functional $J_\lambda$ in $\h$ and in $C^0_s(\overline\Omega)$, respectively (see \cite[Theorem 1.1]{IMS}, \cite[Proposition 2.5]{BCSS}, and \cite[Theorem 1.1]{IMS1} for a nonlinear extension), namely an analog for the fractional case of the main result of \cite{BN1}:

\begin{proposition}\label{svh}
Let ${\bf H}_0$ hold, $u\in\h$. Then, the following are equivalent:
\begin{enumroman}
\item\label{svh1} there exists $\rho>0$ s.t.\ $J_\lambda(u+v)\ge J_\lambda(u)$ for all $v\in\h\cap C^0_s(\overline\Omega)$, $\|v\|_{0,s}\le\rho$;
\item\label{svh2} there exists $\sigma>0$ s.t.\ $J_\lambda(u+v)\ge J_\lambda(u)$ for all $v\in\h$, $\|v\|\le\sigma$.
\end{enumroman}
\end{proposition}

\noindent
As pointed out in the Introduction, we will make use of the following fractional torsion equation on a ball:
\beq\label{T}
\begin{cases}
\fl u_R = 1& \text{in $B_R(x_0)$} \\
u_R=0 & \text{in $B_R(x_0)^c$,}
\end{cases}
\eeq
where $x_0\in\R^N$, $R>0$. The solution of \eqref{T} (defined as in \eqref{ws}) is unique, given by
\[u_R (x) = A(N,s)(R^2 - |x-x_0|^2)_{+}^s, \ A(N,s) = \frac{s \Gamma(\frac{N}{2})} {2 \pi^{\frac{N}{2}} \Gamma(1+s) \Gamma(1-s)}\]
(see \cite[p.\ 33]{BV} or \cite[equation (1.4)]{RS}). This simple example is popular in fractional regularity theory, as it shows that solutions of Dirichlet problems may be singular at the boundary. For future use we compute some norms of $u_R$:

\begin{lemma}\label{tor}
For all $x_0\in\R^N$, $R>0$ we have
\begin{enumroman}
\item\label{tor1} $\displaystyle \|u_R\|_\nu = A(N,s)\Big[\frac{\pi^{\frac{N}{2}} \Gamma(1+\nu s) R^{N+2\nu s}} {\Gamma(\frac{N+2\nu s+2}{2})}\Big]^\frac{1}{\nu}$ for all $\nu\ge 1$;
\item\label{tor2} $\displaystyle[u_R] = \Big[\frac{s\Gamma(\frac{N}{2})R^{N+2s}}{2\Gamma(1-s)\Gamma(\frac{N+2s+2}{2})}\Big]^\frac{1}{2}$.
\end{enumroman}
\end{lemma}
\begin{proof}
First we recall the well-known formulas
\[|\partial B_1(0)|= \frac{2\pi^{\frac{N}{2}}} {\Gamma(\frac{N}{2})}, \quad \int_0^1 (1 - \rho^2)^{\alpha} \rho^{N-1} \,d\rho = \frac{\Gamma(\frac{N}{2}) \Gamma(1+\alpha)} {2\Gamma(\frac{N+2 \alpha+2}{2})} \quad (\alpha >0),\]
then for all $\nu\ge 1$ we compute
\begin{align*}
\int_{B_R(x_0)} u^{\nu}_R(x)\,dx &= A(N,s)^\nu \int_{B_R(x_0)}(R^2 - |x-x_0|^2)^{\nu s} \,dx\\
&= A(N,s)^\nu R^{N+2\nu s} |\partial B_1(0)| \int_0^1 (1 - \rho^2)^{\nu s} \rho^{N-1} \,d\rho \\
&= A(N,s)^\nu \frac{\pi^{\frac{N}{2}} \Gamma(1+\nu s) R^{N+2\nu s}}{\Gamma(\frac{N+2\nu s+2}{2})},
\end{align*}
which implies \ref{tor1}. Further, testing \eqref{T} with $u_R\in H^s_0(B_R(x_0))$ and applying \ref{tor1} with $\nu=1$, we have
\begin{align*}
[u_R]_s^2 &= \int_{B_R(x_0)} u_R\,dx \\
&= A(N,s) \frac{\pi^\frac{N}{2}\Gamma(1+s)R^{N+2s}}{\Gamma(\frac{N+2s+2}{2})} \\
&= \frac{s\Gamma(\frac{N}{2})R^{N+2s}}{2\Gamma(1-s)\Gamma(\frac{N+2s+2}{2})},
\end{align*}
which gives \ref{tor2}.
\end{proof}

\begin{remark}
We note that some results here are affected by the definition \eqref{fl}, which is the same adopted in \cite{BCSS}. Other works on the subject, for instance \cite{BV,CT,DPV}, define the fractional Laplacian as
\[\fl u(x) = C(N,s)\,{\rm P.V.}\int_{\R^N}\frac{u(x)-u(y)}{|x-y|^{N+2s}}\,dy, \ C(N,s) = \frac{2^{2s} s \Gamma(\frac{N+2s}{2})}{\pi^{\frac{N}{2}} \Gamma(1-s)} >0,\]
where the multiplicative constant is required to equivalently define $\fl$ by means of the Fourier transform. In this paper, explicit constants are one of the main issues, so we decide to follow the standard of \cite{BCSS} in order to easily compare similar results.
\end{remark}

\section{Two positive solutions under subcritical growth}\label{sec3} 

\noindent
In this section, following \cite{BD} as a model, we study \eqref{p} under the following hypotheses:

\begin{itemize}[leftmargin=1cm]
\item[${\bf H}_1$] $f \in C(\R)$, $\displaystyle F(t)=\int_0^t f(\tau)\,d\tau$ satisfy
\begin{enumroman}
\item \label{h11} $f(t) \ge 0$ for all $t \in \R$;
\item \label{h12} $f(t) \le a_p |t|^{p-1} + a_q |t|^{q-1}$ for all $t \in \R$ ($1 \le p < 2 <q< 2_s^*$, $a_p, a_q >0$);
\item \label{h13} $\displaystyle\lim_{t \to 0^+} \frac{F(t)}{t^2} = \infty$
\item \label{h14} $0 < \rho F(t) \le f(t)t$ for all $t \ge M$ ($\rho >2$, $M>0$).
\end{enumroman}
\end{itemize}
Hypotheses ${\bf H}_1$ conjure for $f$ a subcritical, superlinear growth at infinity, as well as a sublinear growth near the origin, while \ref{h14} is an Ambrosetti-Rabinowitz condition.
\vskip2pt
\noindent
First, we recall the classical Palais-Smale condition at level $c\in\R$, for a functional $J\in C^1(X)$ on a Banach space $X$:
\begin{itemize}
\item[$(PS)_c$] Every sequence $(u_n)$ in $X$, s.t.\ $J(u_n)\to c$ and $J'(u_n)\to 0$ in $X^*$, has a convergent subsequence.
\end{itemize}
We say that $J$ satisfies $(PS)$, if $J$ satisfies $(PS)_c$ for any $c\in\R$.
\vskip2pt
\noindent
We will apply the following abstract result, slightly rephrased from \cite[Theorem 2.1]{BD}:

\begin{theo} \label{Ab}
Let $X$ be a Banach space, $\Phi, \Psi \in C^1(X)$, $J_{\lambda}=\Phi - \lambda \Psi \; (\lambda>0)$, $r \in \R, \; \bar{u} \in X$ satisfy
\begin{enumerate}[label=$(A_{\arabic*})$]
\item\label{t1} $\displaystyle \inf_{u \in X} \Phi(u) = \Phi(0)=\Psi(0)=0$;
\item\label{t2} $0<\Phi(\bar{u})<r$;
\item\label{t3} $\displaystyle \sup_{\Phi(u) \le r} \cfrac{\Psi(u)}{r} < \cfrac{\Psi(\bar{u})}{\Phi({\bar{u})}}$;
\item\label{t4} $\displaystyle \inf_{u \in X} J_{\lambda}(u) = - \infty$ for all $\displaystyle\lambda \in 
I_r=\Big(\frac{\Phi(\bar{u})}{\Psi(\bar{u})}, \Big[\sup_{\Phi(u) \le r} \frac{\Psi(u)}{r}\Big]^{-1}\Big)$.
\end{enumerate}
Then, for all $\lambda \in I_r$ for which $J_{\lambda}$ satisfies $(PS)$, there exist $u_{\lambda}, v_{\lambda} \in X$ s.t.\ 
\[J_{\lambda}'(u_{\lambda}) = J_{\lambda}'(v_{\lambda})=0, \quad J_{\lambda}(u_{\lambda}) < 0 < J_{\lambda}(v_{\lambda}).\]
\end{theo}

\noindent
Let $T(N,s) >0$ be defined by Lemma \ref{Opt}, set
\beq \label{L}
\lambda^* = \frac{1}{2 T(N,s)^2 |\Omega|^{\frac{2_s^*-2}{2_s^*}}} \Big(\frac{a_p}{p}\Big)^{\frac{2-q}{q-p}} \Big(\frac{a_q}{q}\Big)^{\frac{p-2}{q-p}} \Big(\frac{2-p}{q-2}\Big)^{\frac{2-p}{q-p}} \frac{q-2}{q-p} > 0.
\eeq
We have the following multiplicity result:
\begin{theorem} \label{Subc}
Let ${\bf H}_1$ hold, $\lambda^*>0$ be defined by \eqref{L}. Then, for all $\lambda \in (0, \lambda^*)$, \eqref{p} has at least two solutions 
$u_{\lambda}, v_{\lambda} \in \mathrm{int}(C_s^0(\overline{\Omega})_+)$. 
\end{theorem}

\begin{proof}
Without loss of generality we may assume $f(t) =0$ for all $t \le 0$. We are going to apply Theorem \ref{Ab}. Set $X=\h$ and define $\Phi, \Psi, J_{\lambda}$ as in Section \ref{sec2}, then clearly $\Phi, \Psi \in C^1(\h)$ and
\[\inf_{u \in \h} \Phi(u) = \Phi(0)=\Psi(0)=0,\]
hence hypothesis \ref{t1} holds. Set
\beq \label{R0}
r= \frac{|\Omega|^{\frac{2}{2_s^*}}}{2 T(N,s)^2}  \Big[\frac{a_p q (2-p)}{a_q p (q-2)}\Big]^{\frac{2}{q-p}} >0.
\eeq
For all $u \in \h$, $\Phi(u) \le r$, we have $\|u\| \le (2r)^{\frac{1}{2}}$. So, by hypotheses ${\bf H}_1$ \ref{h11} \ref{h12}, along with \eqref{Ho}, \eqref{L} and \eqref{R0}, we obtain
\begin{align*}
\frac{\Psi(u)}{r} &\le \frac{a_p}{pr} \|u\|_p^p + \frac{a_q}{qr} \|u\|_q^q \\
&\le \frac{a_p}{pr} T(N,s)^p |\Omega|^{\frac{2_s^*-p}{2_s^*}} (2r)^{\frac{p}{2}} + \frac{a_q}{qr} T(N,s)^q |\Omega|^{\frac{2_s^*-q}{2_s^*}} (2r)^{\frac{q}{2}} \\
&= 2 T(N,s)^2  |\Omega|^{\frac{2_s^*-2}{2_s^*}} \Big(\frac{a_p}{p}\Big)^{\frac{q-2}{q-p}} \Big(\frac{a_q}{q}\Big)^{\frac{2-p}{q-p}} \Big(\frac{2-p}{q-2}\Big)^{\frac{p-2}{q-p}} \\
&+ 2 T(N,s)^2  |\Omega|^{\frac{2_s^*-2}{2_s^*}} \Big(\frac{a_p}{p}\Big)^{\frac{q-2}{q-p}}  \Big(\frac{a_q}{q}\Big)^{\frac{2-p}{q-p}} \Big(\frac{2-p}{q-2}\Big)^{\frac{q-2}{q-p}} \\
&= 2 T(N,s)^2  |\Omega|^{\frac{2_s^*-2}{2_s^*}} \Big(\frac{a_p}{p}\Big)^{\frac{q-2}{q-p}}  \Big(\frac{a_q}{q}\Big)^{\frac{2-p}{q-p}}  
\Big(\frac{2-p}{q-2}\Big)^{\frac{p-2}{q-p}}{\frac{q-p}{q-2}} = \frac{1}{\lambda^*}.
\end{align*}
Summarizing,
\beq \label{Sup}
\sup_{\Phi(u) \le r} \frac{\Psi(u)}{r} \le \frac{1}{\lambda^*}.
\eeq
Now fix $\lambda \in (0, \lambda^*)$. Since $\partial \Omega$ is $C^{1,1}$, we can find $x_0 \in \R^N$, $R>0$ largest s.t.\ $B_R(x_0) \subseteq \Omega$.
Let $K>0$ be s.t.\ 
\beq \label{K}
K\frac{s \Gamma(\frac{N}{2}) \Gamma(1+2s) \Gamma(\frac{N+2s+2}{2}) R^{2s}}{\pi^{\frac{N}{2}} \Gamma(1+s)^2 \Gamma(1-s) \Gamma(\frac{N+4s+2}{2})} 
> \frac{1}{\lambda}.
\eeq
By ${\bf H}_1$ \ref{h13}, we can find $\varepsilon >0$ s.t.\ for all $t \in [0,\varepsilon]$
\beq \label{Ft}
F(t) \ge K t^2.
\eeq
Finally, fix
\beq \label{D}
0 < \delta < \min \Big\{\Big[\frac{4 \Gamma(1-s) \Gamma(\frac{N+2s+2}{2}) r} {s \Gamma(\frac{N}{2}) R^{N+2s}}\Big]^{\frac{1}{2}}, 
\frac{2 \pi^{\frac{N}{2}} \Gamma(1+s) \Gamma(1-s) \varepsilon}{s \Gamma(\frac{N}{2})R^{2s}}\Big\}.
\eeq
Now let $u_R$ be the solution of \eqref{T} in $B_R(x_0)$, and set $\bar{u}=\delta u_R \in \h$. Then we have by Lemma \ref{tor} \ref{tor2} and \eqref{D}
\[
\Phi(\bar{u})= \frac{s \Gamma(\frac{N}{2}) R^{N+2s} \delta^2}{4 \Gamma(1-s) \Gamma(\frac{N+2s+2}{2})} < r,
\]
which implies \ref{t2}. Besides, by \eqref{D} we have for all $x \in \Omega$
\[
0 \le \bar{u}(x) \le \frac{s \Gamma(\frac{N}{2}) R^{2s} \delta}{2 \pi^{\frac{N}{2}} \Gamma(1+s) \Gamma(1-s)} < \varepsilon,
\]
hence by \eqref{Ft} and Lemma \ref{tor} \ref{tor1}
\[
\Psi(\bar{u}) \ge \int_{\Omega} K \bar{u}^2\,dx = K \delta^2 \|u_R\|_2^2 = K\frac{s^2 \Gamma(\frac{N}{2})^2 \Gamma(1+2s) R^{N+4s}}{4 \pi^{\frac{N}{2}} \Gamma(1+s)^2 \Gamma(1-s)^2 \Gamma(\frac{N+4s+2}{2})}\delta^2.
\]
The relations above and \eqref{K} imply
\[
\frac{\Psi(\bar{u})} {\Phi(\bar{u})} \ge K\frac{s \Gamma(\frac{N}{2}) \Gamma(1+2s)\Gamma(\frac{N+2s+2}{2})R^{2s}} {\pi^{\frac{N}{2}} \Gamma(1+s)^2 \Gamma(1-s) \Gamma(\frac{N+4s+2}{2})}
> \frac{1}{\lambda}.
\]
Recalling that $\lambda<\lambda^*$, by \eqref{Sup} we have
\[
\sup_{\Phi(u) \le r} \frac{\Psi(u)}{r} < \frac{1}{\lambda} < \frac{\Psi(\bar{u})}{\Phi(\bar{u})},
\]
which yields at once \ref{t3} and $\lambda \in I_r$. By ${\bf H}_1$ \ref{h14} we can find $C>0$ s.t.\ for all $t \ge M$
\beq \label{F}
F(t) \ge C t^{\rho}.
\eeq
Now pick $w \in C^\infty_c(\Omega)\setminus \{0\}$. By \eqref{F}, and recalling that $F(t)\ge 0$ for all $t\in\R$, we have for all $\tau >0$
\begin{align*}
J_{\lambda}(\tau w) &\le \frac{\|w\|^2}{2} \tau^2 - \lambda \int_{\{w \le M/\tau\}} F(\tau w)\,dx - \lambda \int_{\{w > M/\tau\}} C(\tau w)^{\rho}\,dx\\
&\le \frac{\|w\|^2}{2} \tau^2 - \lambda \int_{\Omega} C (\tau w)^{\rho}\,dx + \lambda \int_{\{w \le M/\tau\}} C M^{\rho}\,dx\\
&\le \frac{\|w\|^2}{2}  \tau^2 - \lambda C \|w\|_{\infty}^{\rho} |\Omega| \tau^{\rho} + \lambda CM^{\rho} |\Omega|
\end{align*}
and the latter tends to $-\infty$ as $\tau \to \infty$ (since $\rho >2$). So we see that \ref{t4} holds as well.
\vskip2pt
\noindent
Finally, we prove that $J_{\lambda}$ satisfies $(PS)$. Let $(u_n)$ be a sequence in $\h$ s.t.\ 
$|J_{\lambda}(u_n)| \le C$, $J_{\lambda}'(u_n) \to 0$ in $H^{-s}(\Omega)$. Then, for all $n \in \N$ we have
\beq \label{B}
\frac{\|u_n\|^2}{2}  - \lambda \int_{\Omega} F(u_n)\,dx \le C
\eeq
and for all $\varphi \in \h$
\beq \label{Lim}
\Big|\langle u_n, \varphi \rangle - \lambda \int_{\Omega} f(u_n)\varphi\,dx\Big| \le \|J_{\lambda}'(u_n)\| \, \|\varphi\|
\eeq
Multiplying \eqref{B} by $\rho>2$ (as in ${\bf H}_1$ \ref{h14}), testing \eqref{Lim} with $u_n$, and subtracting,
\begin{align*}
\frac{\rho-2}{2} \|u_n\|^2 &\le \lambda \int_{\Omega} \big(\rho F(u_n)-f(u_n)u_n\big)\,dx + \|J_{\lambda}'(u_n)\| \, \|u_n\| + C \\
&\le \lambda \int_{\{0 \le u_n \le M\}} C \big(|u_n|^p + |u_n|^q\big)\,dx + \|J_{\lambda}'(u_n)\| \, \|u_n\| + C\\
& \le \lambda C (M^p + M^q) |\Omega| + \|J_{\lambda}'(u_n)\| \, \|u_n\| + C.
\end{align*}
So $(u_n)$ is bounded in $\h$. Passing to a subsequence, we have $u_n \rightharpoonup u$ in $\h$, $u_n \to u$ in $L^p(\Omega), L^q(\Omega)$, and $u_n(x) \to u(x)$ for a.e. $x \in \Omega$. Testing \eqref{Lim} this time with $u_n-u \in \h$, we have for all $n \in \N$
\begin{align*}
\|u_n-u\|^2 &\le \langle u, u_n-u\rangle + \lambda \int_{\Omega} \big(a_p |u_n|^{p-1} + a_q |u_n|^{q-1}\big)|u_n-u|\,dx + \|J_{\lambda}'(u_n)\| \, \|u_n-u\|\\
&\le \langle u, u_n-u\rangle + \lambda\big(a_p \|u_n\|_p^{p-1} \|u_n-u\|_p + a_q \|u_n\|_q^{q-1} \|u_n-u\|_q\big) + \|J_{\lambda}'(u_n)\| \, \|u_n-u\|,
\end{align*}
(where we used ${\bf H}_1$ \ref{h12} and H\"{o}lder's inequality), and the latter tends to $0$ as $n\to \infty$. So, $u_n \to u$ in $\h$. (Note that we actually proved that $J_\lambda$ is unbounded from below and satisfies $(PS)$ for {\em all} $\lambda>0$.)
\vskip2pt
\noindent
By Theorem \ref{Ab}, there exist $u_{\lambda}, v_{\lambda} \in \h$ s.t.\
\[
J_{\lambda}'(u_{\lambda})= J_{\lambda}'(v_{\lambda})= 0, \quad J_{\lambda}(u_{\lambda})< 0 < J_{\lambda}(v_{\lambda}).
\]
Therefore, $u_{\lambda}, v_{\lambda} \not\equiv 0$ solve \eqref{p}. By ${\bf H}_0$ \ref{h11} and Proposition \ref{reg}, finally, we have $u_{\lambda}, v_{\lambda} \in \mathrm{int}(C_s^0(\overline{\Omega})_+)$.
\end{proof}

\noindent
We focus now on problem \eqref{pp}, with $1<p<2<q<2_s^*$ (subcritical case) and $\mu>0$. Set
\beq \label{Mu}
\mu^*= \Big[2 T(N,s)^2 |\Omega|^{\frac{2_s^*-2}{2_s^*}}\Big]^{\frac{p-q}{q-2}} p \, q^{\frac{2-p}{q-2}} \Big(\frac{2-p}{q-2}\Big)^{\frac{2-p}{q-2}} 
\Big(\frac{q-2}{q-p}\Big)^{\frac{q-p}{q-2}} >0.
\eeq
We have the following multiplicity result:

\begin{corollary} \label{Pt}
Let $1<p<2<q<2_s^*$, $\mu^*>0$ be defined by \eqref{Mu}. Then, for all $\mu \in (0, \mu^*)$ \eqref{pp} has at least two solutions 
$u_{\mu}, v_{\mu} \in \mathrm{int}(C_s^0(\overline{\Omega})_+)$.
\end{corollary}
\begin{proof}
Set for all $t \in \R$, $\mu \in (0,\mu^*)$
\[f(t)=\mu (t^+)^{p-1} + (t^+)^{q-1}.\]
Then $f$ satisfies ${\bf H}_1$ with $a_p=\mu$, $a_q=1$, and any $\rho\in(2,q)$. In view of \eqref{Mu}, here \eqref{L} rephrases as
\[
\lambda^* = \frac{1}{2 T(N,s)^2 |\Omega|^{\frac{2_s^*-2}{2_s^*}}} p^{\frac{q-2}{q-p}} q^{\frac{2-p}{q-p}} \Big(\frac{2-p}{q-2}\Big)^{\frac{2-p}{q-p}} 
\frac{q-2}{q-p} \,  \mu^{\frac{2-q}{q-p}} > 1.
\]
Hence we can apply Theorem \ref{Subc} with $\lambda=1$ and find $u_{\mu}, v_{\mu} \in \mathrm{int}(C_s^0(\overline{\Omega})_+)$ solutions to \eqref{pp}.
\end{proof}

\noindent
We present an example for Corollary \ref{Pt}:

\begin{example}
Set $s=\frac{1}{2}, \, p=\frac{3}{2}, \, q=3, \, N=2$ and
\[ \Omega = \Big\{(x,y) \in \R^2: \frac{x^2}{4} + \frac{y^2}{9} \le 1\Big\}. \]
Then we have $2_{1/2}^*=4 >3, \, |\Omega| = 6 \pi,$ while Lemma \ref{Opt} gives
\[T\Big(2, \frac{1}{2}\Big)= \frac{(\frac{1}{2})^{\frac{1}{2}} \Gamma(\frac{1}{2})^{\frac{1}{2}} \Gamma(2)^{\frac{1}{4}}} 
{2^{\frac{1}{2}} \pi^{\frac{3}{4}} \Gamma(\frac{1}{2})^{\frac{1}{2}} \Gamma(1)^{\frac{1}{4}}} = \frac{1}{2 \pi^{\frac{3}{4}}}. 
\]
Therefore, \eqref{Mu} becomes
\[
\mu^*= \Big[2 \Big(\frac{1}{2 \pi^{\frac{3}{4}}}\Big)^2 (6 \pi)^{\frac{1}{2}}\Big]^{-\frac{3}{2}} \frac{3}{2} \, 3^{\frac{1}{2}} \Big(\frac{1}{2}\Big)^{\frac{1}{2}}
\Big(\frac{2}{3}\Big)^{\frac{3}{2}} = \frac{2^{\frac{3}{4}} \pi^{\frac{3}{2}}}{3^{\frac{3}{4}}} .
\]
By Corollary \ref{Pt}, for all $\mu \in (0,\mu^*)$ \eqref{pp} has at least two positive solutions.
\end{example}

\section{One positive solution under critical growth}\label{sec4} 

\noindent
In this section, we study the following slight generalization of problem \eqref{pp}:
\beq\label{Cr}
\begin{cases}
\fl u = \mu g(u) + u^{2_s^*-1} & \text{in $\Omega$} \\
u > 0 & \text{in $\Omega$} \\
u=0 & \text{in $\Omega^c$,}
\end{cases}
\eeq
with $\mu>0$ and assuming the following hypotheses on $g$:
\begin{itemize}[leftmargin=1cm]
\item[${\bf H}_2$] $g \in C(\R)$, $\displaystyle G(t)=\int_0^t g(\tau)\,d\tau$ satisfy
\begin{enumroman}
\item\label{h21} $g(t) \ge 0$ for all $t \in \R$;
\item\label{h22} $g(t) \le a_p |t|^{p-1}$ for all $t \in \R$ ($p \in (1, 2_s^*)$, $a_p >0$);
\item\label{h23} $\displaystyle\lim_{t \to 0^+} \frac{G(t)}{t^2} = \infty$.
\end{enumroman}
\end{itemize}
Note that, due to hypothesis ${\bf H}_1$ \ref{h23}, problem \eqref{Cr} reduces to \eqref{pp} with $g(t)=t^{p-1}$ only for $p\in(1,2)$ (concave case). Although, the results of this section also embrace the case $p\in[2,2^*_s)$ (linear/convex case).
\vskip2pt
\noindent
Due to the presence of the critical term $u^{2^*_s-1}$ in \eqref{Cr}, we cannot apply Theorem \ref{Ab}, as the associated energy functional does not satisfy $(PS)$ in general. So we introduce the following local Palais-Smale condition for functionals of the type $J_\lambda=\Phi-\lambda\Psi$, with $\Phi,\Psi\in C^1(X)$, $\lambda>0$, defined on a Banach space $X$, and $r>0$:
\begin{itemize}
\item[$(PS)^r$] Every sequence $(u_n)$ in $X$, s.t.\ $(J_\lambda(u_n))$ is bounded in $\R$, $J'(u_n)\to 0$ in $X^*$, and $\Phi(u_n)\le r$ for all $n\in\N$, has a convergent subsequence.
\end{itemize}
In this case, our main tool is the following local minimum result, slightly rephrased from \cite[Theorem 3.3]{BDO}:

\begin{theo}\label{Ab1}
Let $X$ be a Banach space, $\Phi, \Psi \in C^1(X)$, $J_{\lambda}=\Phi - \lambda \Psi$ ($\lambda >0$), $r \in \R$; $\bar{u} \in X$ satisfy
\begin{enumerate}[label=(\subscript{B}{{\arabic*}})]
\item \label{b1} $\displaystyle \inf_{u \in X} \Phi(u) = \Phi(0)=\Psi(0)=0$;
\item \label{b2} $0<\Phi(\bar{u})<r$;
\item \label{b3} $\displaystyle \sup_{\Phi(u) \le r} \cfrac{\Psi(u)}{r} < \cfrac{\Psi(\bar{u})}{\Phi({\bar{u})}}$.
\end{enumerate}
Let
\[I_r=\Big(\frac{\Phi(\bar{u})}{\Psi(\bar{u})}, \Big[\displaystyle \sup_{\Phi(u) \le r} \frac{\Psi(u)}{r}\Big]^{-1}\Big).\]
Then, for all $\lambda \in I_r$ for which $J_{\lambda}$ satisfies $(PS)^r$, there exists $u_{\lambda} \in X$ s.t.\
\[0<\Phi(u_{\lambda})< r, \quad J_{\lambda}(u_{\lambda})=\min_{0<\Phi(u)<r} J_{\lambda}(u).\]
\end{theo}

\noindent
Set for all $\mu>0$, $t\in\R$
\[f(t) = \mu g(t)+(t^+)^{2^*-1}, \ F(t) = \int_0^t f(\tau)\,d\tau,\]
then define $\Phi,\Psi\in C^1(\h)$ as in Section \ref{sec2}. Further, for all $\lambda>0$ set $J_\lambda=\Phi-\lambda\Psi$. Set for all $r,\mu >0$ 
\beq \label{Min}
\lambda_r^* = \min \Big\{\Big[\frac{2^{\frac{2_s^*}{2}}T(N,s)^{2_s^*} r^{\frac{2_s^*-2}{2}}}{2_s^*} + \mu \frac{2^{\frac{p}{2}} a_p T(N,s)^p |\Omega|^{\frac{2_s^*-p}{2_s^*}} 
r^{\frac{p-2}{2}}}{p}\Big]^{-1}, \frac{1}{T(N,s)^{2_s^*}} \Big[\frac{s}{2Nr}\Big]^{\frac{2s}{N-2s}}\Big\}
\eeq
We prove now that $J_\lambda$ satisfies $(PS)^r$ for all $r>0$ and all $\lambda>0$ small enough:

\begin{lemma}\label{PS}
Let $r, \mu >0, \, \lambda_r^*>0$ be defined by \eqref{Min}. Then $J_{\lambda}$ satisfies $(PS)^r$ for all $\lambda \in (0,\lambda_r^*)$.
\end{lemma}
\begin{proof}
Let $(u_n)$ be a sequence in $\h$ s.t.\ $(J_{\lambda}(u_n))$ is bounded, $J_{\lambda}'(u_n) \to 0$ in $H^{-s}(\Omega)$, and $\Phi(u_n)\le r$ for all $n \in \N$.
Then $(u_n)$ is bounded in $\h$, hence in $L^{2_s^*}(\Omega)$ (Lemma \ref{Opt}). Passing to a subsequence we have $u_n \rightharpoonup u$ in $\h$, 
$L^{2_s^*}(\Omega)$, $u_n \to u$ in $L^p(\Omega)$, $u_n(x) \to u(x)$ for a.e. $x \in \Omega$, and $J_{\lambda}(u_n) \to c$.
\vskip2pt
\noindent
First we see that
\beq \label{C1}
J_{\lambda}'(u) = 0.
\eeq
Indeed, since $(u_n^{2_s^*-1})$ is bounded in $L^{(2_s^*)'}(\Omega)$, up to a further subsequence we have $u_n^{2_s^*-1} \rightharpoonup u^{2_s^*-1}$ in
$L^{(2_s^*)'}(\Omega)$, while by ${\bf H}_2$ \ref{h21} \ref{h22} we have $g(u_n) \to g(u)$ in $L^{p'}(\Omega)$. So, for all $\varphi \in \h$ we have
\begin{align*}
\langle J_{\lambda}'(u_n), \varphi \rangle &= \langle u_n, \varphi \rangle - \lambda \int_{\Omega} u_n^{2_s^*-1} \varphi\,dx - \lambda \mu \int_{\Omega} g(u_n)\varphi\,dx\\
& \to \langle u, \varphi \rangle - \lambda \int_{\Omega} u^{2_s^*-1} \varphi\,dx - \lambda \mu \int_{\Omega} g(u)\varphi\,dx= \langle J_{\lambda}'(u), \varphi \rangle,
\end{align*}
which along with $J_{\lambda}'(u_n) \to 0$ gives \eqref{C1}. Besides,
\beq \label{C2}
J_{\lambda}(u) > -r.
\eeq
Indeed, since $u_n \rightharpoonup u$ in $\h$ and $\Phi$ is convex, we have $\Phi(u) \le r$, i.e., $\|u\| \le (2r)^{\frac{1}{2}}$. So using Lemma \ref{Opt}, \eqref{Ho} with $\nu=p$, \eqref{Min}, and 
$\lambda < \lambda_r^*$, we have
\begin{align*}
J_{\lambda}(u) &\ge -\lambda \Psi(u) \\
&\ge - \frac{\lambda}{2_s^*} \|u\|_{2_s^*}^{2_s^*} - \frac{\lambda \mu a_p}{p} \|u\|_{p}^{p} \\
&\ge - \frac{\lambda}{2_s^*} T(N,s)^{2_s^*} (2r)^{\frac{2_s^*}{2}} - \frac{\lambda \mu a_p}{p} T(N,s)^p |\Omega|^{\frac{2_s^*-p}{2_s^*}} (2r)^{\frac{p}{2}} \\
&\ge -\lambda r \Big[\frac{2^{\frac{2_s^*}{2}}T(N,s)^{2_s^*} r^{\frac{2_s^*-2}{2}}}{2_s^*} + \mu \frac{2^{\frac{p}{2}}a_pT(N,s)^p |\Omega|^{\frac{2_s^*-p}{2_s^*}} r^{\frac{p-2}{2}}}{p}\Big] \\ 
&\ge - \frac{\lambda r}{\lambda_r^*},
\end{align*}
and the latter gives \eqref{C2} since $\lambda>\lambda^*_r$. Now set $v_n=u_n-u$. We have
\beq \label{C3}
\lim_n \Big[\Phi(v_n) - \frac{\lambda}{2_s^*} \|v_n\|_{2_s^*}^{2_s^*} \Big] = c - J_{\lambda}(u).
\eeq
Indeed, since $v_n \rightharpoonup 0$ in $\h$, we have
\[
\|v_n\|^2 = \|u_n\|^2 - 2 \langle u_n, u \rangle + \|u\|^2 = \|u_n\|^2 - \|u\|^2 + o(1)
\]
(as $n\to\infty$). Since $v_n \rightharpoonup 0$ in $L^{2_s^*}(\Omega)$, by the Brezis-Lieb Lemma \cite[Theorem 1]{BL} we have 
\[\|v_n\|_{2_s^*}^{2_s^*}= \|u_n\|_{2_s^*}^{2_s^*} - \|u\|_{2_s^*}^{2_s^*} + o(1).\]
Since $u_n \to u$ in $L^p(\Omega)$, we have $G(u_n) \to G(u)$ in $L^1(\Omega)$. So,
\begin{align*}
\Phi(v_n) - \frac{\lambda}{2_s^*} \|v_n\|_{2_s^*}^{2_s^*} &= [\Phi(u_n) -\Phi(u)] - \frac{\lambda}{2_s^*} \big[\|u_n\|_{2_s^*}^{2_s^*} - \|u\|_{2_s^*}^{2_s^*}\big]
- \lambda \mu  \int_{\Omega} [G(u_n) - G(u)]\,dx + o(1)\\
&=J_{\lambda}(u_n) - J_{\lambda}(u) + o(1) \to c - J_{\lambda}(u). 
\end{align*}
On the other hand,
\beq \label{C4}
\lim_n \big[\|v_n\|^2 - \lambda \|v_n\|_{2_s^*}^{2_s^*} \big] = 0.
\eeq
Indeed, arguing as above and recalling that $g(u_n) u_n \to g(u)u$ in $L^1(\Omega)$, we have
\begin{align*}
\|v_n\|^2 - \lambda \|v_n\|_{2_s^*}^{2_s^*} &= [\|u_n\|^2 -\|u\|^2] - \lambda\big[\|u_n\|_{2_s^*}^{2_s^*} - \|u\|_{2_s^*}^{2_s^*}\big]
- \lambda \mu  \int_{\Omega} [g(u_n)u_n - g(u)u]\,dx + o(1)\\
&=\langle J_{\lambda}'(u_n), u_n \rangle - \langle J_{\lambda}'(u), u \rangle + o(1),
\end{align*}
and the latter tends to $0$ as $n \to \infty$, by $J_{\lambda}'(u_n)\to 0$, boundedness of $(u_n)$, and \eqref{C1}.
Recalling that $(v_n)$ is bounded in $\h$, up to a subsequence we have $\|v_n\| \to \beta \ge 0$. We prove that
\beq \label{C5}
\beta =0,
\eeq
arguing by contradiction. Assume $\beta >0$. Then, by \eqref{C4} we have
\[
\beta^2 = \lim_n \lambda \|v_n\|_{2_s^*}^{2_s^*} \le \lambda T(N,s)^{2_s^*} \beta^{2_s^*},
\]
hence 
\[
\beta \ge \Big[\frac{1}{\lambda T(N,s)^{2_s^*}}\Big]^{\frac{1}{2_s^*-2}}.
\]
By \eqref{C2} and \eqref{C3}, we also have
\[
\Big(\frac{1}{2} - \frac{1}{2_s^*} \Big) \beta^2 = c - J_{\lambda}(u) < 2r,
\]
hence
\[
\beta < \Big[\frac{2Nr}{s}\Big]^{\frac{1}{2}}.
\]
Comparing the last inequalities and recalling \eqref{Min}, we get
\[
\lambda > \frac{1}{T(N,s)^{2_s^*}} \Big[\frac{s}{2Nr}\Big]^{\frac{2s}{N-2s}} \ge \lambda_r^*,\]
a contradiction. So \eqref{C5} is proved, which means $u_n \to u$ in $\h$. Thus, $J_\lambda$ satisfies $(PS)^r$.
\end{proof}

\noindent
Set 
\beq \label{mS}
\mu^* = \min \Big\{\Big[\frac{2_s^*}{2^{\frac{2_s^* + 2}{2}}T(N,s)^{2_s^*}}\Big]^{\frac{2}{2_s^*-2}}, \frac{s}{3NT(N,s)^{\frac{N}{s}}}\Big\}^{\frac{2-p}{2}} 
\frac{p}{2^{\frac{p+2}{2}} a_p T(N,s)^p |\Omega|^{\frac{2_s^*-p}{2_s^*}}} >0.
\eeq
We have the following existence result for problem \eqref{Cr}:

\begin{theorem} \label{Crit}
Let ${\bf H}_2$ hold, $\mu^*>0$ be defined by \eqref{mS}. Then, for all $\mu \in (0, \mu^*)$, \eqref{Cr} has at least one solution
$u_{\mu} \in \mathrm{int}(C_s^0(\overline{\Omega})_+)$. 
\end{theorem}
\begin{proof}
Fix $\mu \in (0, \mu^*)$ and set
\beq \label{R}
r = \min \Big\{\Big[\frac{2_s^*}{2^{\frac{2_s^* + 2}{2}}T(N,s)^{2_s^*}}\Big]^{\frac{2}{2_s^*-2}}, \frac{s}{3 NT(N,s)^{\frac{N}{s}}}\Big\} >0.
\eeq
By \eqref{mS}, \eqref{R} we have 
\[
\frac{2^{\frac{2_s^*}{2}}T(N,s)^{2_s^*} r^{\frac{2_s^*-2}{2}}}{2_s^*} + \mu \frac{2^{\frac{p}{2}} a_p T(N,s)^p |\Omega|^{\frac{2_s^*-p}{2_s^*}} r^{\frac{p-2}{2}}}{p}
\le \frac{1}{2} + \frac{\mu}{2\mu^*} < 1,
\]
as well as
\[
\frac{1}{T(N,s)^{2_s^*}} \Big[\frac{s}{2Nr}\Big]^{\frac{2s}{N-2s}} \ge \frac{1}{T(N,s)^{2^*_s}}\Big[\frac{s}{2N}\frac{3NT(N,s)^\frac{N}{s}}{s}\Big]^\frac{2s}{N-2s} = \Big(\frac{3}{2}\Big)^{\frac{2s}{N-2s}} >1,
\]
hence by \eqref{Min} we have $\lambda_r^* >1$.
\vskip2pt
\noindent
We intend to apply Theorem \ref{Ab1}. First, we see that hypothesis \ref{b1} holds. Then, for all $u \in \h$, $\Phi(u) \le r$ we have by ${\bf H}_2$ \ref{h21} \ref{h22}, Lemma \ref{Opt}, and \eqref{Ho}
\begin{align*}
\frac{\Psi(u)}{r} &\le \frac{\|u\|_{2_s^*}^{2_s^*}}{2_s^* r}  +\mu  \frac{a_p \|u\|_p^p}{pr} \\
&\le \frac{T(N,s)^{2_s^*} (2r)^{\frac{2_s^*}{2}}}{2_s^* r}  +\mu  \frac{a_p T(N,s)^p |\Omega|^{\frac{2_s^*-p}{2_s^*}} (2r)^{\frac{p}{2}}}{pr} \\
&\le \frac{1}{\lambda^*_r}.
\end{align*}
On the other hand, by ${\bf H}_2$ \ref{h23} we have
\[
\lim_{t \to 0^+} \frac{F(t)}{t^2} = \infty.
\]
So, arguing as in the proof of Theorem \ref{Subc}, we can find $\bar{u} \in \h$ s.t.\
\[
0<\Phi(\bar{u})<r, \quad \frac{\Psi(\bar{u})}{\Phi({\bar{u})}} > \frac{1}{\lambda^*_r},
\]
which ensures \ref{b2} and \ref{b3}. Finally, since $\lambda^*_r>1$, by Lemma \ref{PS} the functional $J_1$ satisfies $(PS)^r$.
\vskip2pt
\noindent
Since $1\in I_r$, from Theorem \ref{Ab1} we deduce the existence of a (relabeled) function $u_\mu\in\h$ s.t.\
\[0<\Phi(u_\mu)<r, \ J_1(u_\mu) = \min_{0<\Phi(u_\mu)<r}J_1(u).\]
In particular, we have $J'_1(u_\mu)=0$ in $H^{-s}(\Omega)$. Thus, by Proposition \ref{reg}, $u_\mu\in{\rm int}(C^0_s(\overline\Omega)_+)$ is a solution of \eqref{Cr}.
\end{proof}

\begin{remark}\label{Crit1}
The proof of Theorem \ref{Crit} gives additional information: $u_{\mu}$ is a local minimizer of $J_1$ in $\h$, satisfies the bound $\|u_{\mu}\| < (2r)^{\frac{1}{2}}$, and the mapping 
$\mu \mapsto J_1(u_{\mu})$ is decreasing in $(0,\mu^*)$.
\end{remark}

\section{Two positive solutions under critical growth}\label{sec5}

\noindent
Finally, we turn to problem \eqref{pp} with $q=2^*_s$, namely, the Brezis-Nirenberg problem for the fractional Laplacian:
\beq \label{bn}
\begin{cases}
\fl u = \mu u^{p-1} + u^{2^*_s-1} & \text{in $\Omega$} \\
u > 0 & \text{in $\Omega$} \\
u=0 & \text{in $\Omega^c$,}
\end{cases}
\eeq
with $p\in(1,2)$, $\mu>0$. This is a special case of \eqref{Cr} with $g(t)=(t^+)^{p-1}$, which satisfies ${\bf H}_2$ with $a_p=1$. We know from \cite[Theorem 1.1]{BCSS} that \eqref{bn} has at least two positive solutions for all $\mu>0$ small enough. Our last result yields an explicitly estimate of 'how small' $\mu$ should be, given by \eqref{mS} which in the present case rephrases as
\beq \label{muc}
\mu^* = \min \Big\{\Big[\frac{2_s^*}{2^{\frac{2_s^* + 2}{2}}T(N,s)^{2_s^*}}\Big]^{\frac{2}{2_s^*-2}}, \frac{s}{3NT(N,s)^{\frac{N}{s}}}\Big\}^{\frac{2-p}{2}} \frac{p}{2^{\frac{p+2}{2}}T(N,s)^p |\Omega|^{\frac{2_s^*-p}{2_s^*}}} >0.
\eeq
Indeed, we have the following multiplicity result:

\begin{theorem}\label{mc}
Let $p\in(1,2)$, $\mu^*>0$ be defined by \eqref{muc}. Then, for all $\mu \in (0, \mu^*)$, \eqref{bn} has at least two solutions $u_{\mu}, w_{\mu} \in \mathrm{int}(C_s^0(\overline{\Omega})_+)$, $u_{\mu} < w_{\mu}$ in $\Omega$. 
\end{theorem}
\begin{proof}
Fix $\mu\in(0,\mu^*)$, define $f\in C(\R)$, $\Phi,\Psi\in C^1(\h)$ as in Section \ref{sec4}, and set for brevity $J=J_1=\Phi-\Psi$. From Theorem \ref{Crit} and Remark \ref{Crit1} we know that there exists $u_\mu\in\h\cap {\rm int}(C^0_s(\overline\Omega)_+)$ which solves \eqref{bn} and is a local minimizer of $J$. Set for all $(x,t)\in\Omega\times\R$
\[\tilde{f}(x,t)=f(u_{\mu}(x) + t^+)-f(u_{\mu}(x)),\]
\[\tilde{F}(x,t)=\int_0^t \tilde{f}(x,\tau)\,d\tau=F(u_{\mu}(x) + t^+)-F(u_{\mu}(x))-f(u_{\mu}(x))t^+.\]
For all $v\in\h$ set
\[\tilde\Psi(v) = \int_\Omega \tilde{F}(x,v)\,dx, \quad \tilde{J}(v)=\Phi(v)-\tilde{\Psi}(v).\]
As in Section \ref{sec2}, it is easily seen that $\tilde{J}\in C^1(\h)$ and all its critical points solve the (nonautonomous) auxiliary problem
\beq \label{Cr2}
\begin{cases}
\fl v = \tilde{f}(x,v) & \text{in $\Omega$} \\
v=0 & \text{in $\Omega^c$.}
\end{cases}
\eeq
The functionals $\tilde J$ and $J$ are related to each other by the following inequality for all $v\in\h$:
\beq\label{in}
\tilde J(v) \ge J(u_\mu+v^+)-J(u_\mu)+\frac{\|v^-\|^2}{2}.
\eeq
Indeed, we have $v^\pm\in\h$ and, setting
\[\Omega_+=\big\{x \in \Omega: v(x)>0\big\}, \quad \Omega_-= \Omega \setminus \Omega_+,\]
from $v=v^+-v^-$ we have
\begin{align*}
\|v\|^2 &= \|v^+\|^2 + \|v^-\|^2 - 2 \iint_{\R^N\times\R^N}\frac{(v^+(x)-v^+(y))(v^-(x)-v^-(y))}{|x-y|^{N+2s}}\,dx\,dy \\
&\ge \|v^+\|^2 + \|v^-\|^2,
\end{align*}
as the integrand vanishes everywhere but in $\Omega_+\times\Omega_-$ and in $\Omega_-\times\Omega_+$, where is is negative. So we have
\begin{align*}
\tilde J(v) &= \frac{\|v\|^2}{2} - \int_{\Omega} \tilde{F}(x,v)\,dx \\
&\ge \frac{\|v^+\|^2}{2} + \frac{\|v^-\|^2}{2} - \int_{\Omega} \big[F(u_{\mu}+v^+) - F(u_{\mu})-f(u_{\mu})v^+\big]\,dx \\
&= \frac{\|u_{\mu} + v^+\|^2}{2} - \frac{\|u_{\mu}\|^2}{2} - \langle u_{\mu}, v^+ \rangle +  \frac{\|v^-\|^2}{2} - \int_{\Omega} \big[F(u_{\mu}+v^+) - F(u_{\mu})-f(u_{\mu})v^+\big]\,dx \\
&= J(u_{\mu} + v^+) - J(u_{\mu}) + \frac{\|v^-\|^2}{2}
\end{align*}
(where we used that $u_\mu$ solves \eqref{bn}).
\vskip2pt
\noindent
We claim that $0$ is a local minimizer of $\tilde{J}$. Indeed, by Proposition \ref{svh} there exists $\rho>0$ s.t.\ for all $v\in\h\cap C^0_s(\overline\Omega)$, $\|v\|_{0,s}\le\rho$ we have $J(u_\mu+v)\ge J(u_\mu)$. Then, for any such $v$ we have as well $\|v^+\|_{0,s}\le\rho$, which along with \eqref{in} implies
\[\tilde J(v) \ge J(u_\mu+v^+)-J(u_\mu)+\frac{\|v^-\|^2}{2} \ge 0 = \tilde J(0).\]
So, $0$ is a local minimizer of $\tilde J$ in $C^0_s(\overline\Omega)$ and hence, by Proposition \ref{svh} again, it is such also in $\h$. In particular, $\tilde J'(0)=0$ in $H^{-s}(\Omega)$, i.e., $0$ solves \eqref{Cr2}.
\vskip2pt
\noindent
From now on we closely follow \cite{BCSS}. Arguing by contradiction, assume that $0$ is the {\em only} critical point of $\tilde J$ in $\h$. Under such assumption, by \cite[Lemma 2.10]{BCSS} $\tilde{J}$ satisfies $(PS)_c$ at any level $c<c^*$, where
\beq \label{C}
c^*=\frac{s}{NT(N,s)^{\frac{N}{s}}}.
\eeq
Fix $x_0\in\Omega$, and for all $\varepsilon>0$, define $v_{\varepsilon} \in H^s(\R^N)$ by setting for all $x \in \R^N$
\[v_{\varepsilon}(x)=\frac{\varepsilon^{\frac{N-2s}{2}}}{(\varepsilon^2 + |x-x_0|^2)^{\frac{N-2s}{2}}}.\]
By Lemma \ref{Opt} we have
\beq \label{S}
\|v_{\varepsilon}\|_{2_s^*} = T(N,s) [v_{\varepsilon}]_s.
\eeq
Now fix $r>0$ s.t.\ $\overline{B}_r(x_0) \subset \Omega$, $\eta \in C^{\infty}(\R^N)$ s.t.\ $\eta =1$ in $B_{\frac{r}{2}}(x_0)$, $\eta=0$ in $B_1^c(x_0)$, and $0 \le \eta \le 1$ in $\R^N$, then define $w_{\varepsilon} \in \h$ by setting for all $x\in\R^N$
\[w_{\varepsilon}(x)= \frac{\eta(x) v_{\varepsilon}(x)}{\|\eta v_{\varepsilon}\|_{2_s^*}}.\]
Clearly $\|w_{\varepsilon}\|_{2_s^*}=1$. Besides, we will prove that for all $\varepsilon >0$ small enough
\beq \label{Max}
\max_{\tau\ge 0} \tilde{J}(\tau w_{\varepsilon}) < c^*.
\eeq
Assume $N>4s$. Then, by \cite[Propositions 21, 22]{SV} we find for all $\varepsilon >0$ small enough
\begin{align*}
&\|w_{\varepsilon}\|^2 \le \frac{1}{T(N,s)^2} + C {\varepsilon}^{N-2s}\\
&\|w_{\varepsilon}\|_2^2 \ge C {\varepsilon}^{2s} - C {\varepsilon}^{N-2s}
\end{align*}
 ($C>0$ denotes several constants, independent of $\varepsilon$). By convexity we have for all $x\in\Omega$, $t\ge 0$
\[\tilde{F}(x,t) \ge \frac{t^{2_s^*}}{2_s^*} + \frac{C}{2} u_{\mu}(x)^{2_s^* -2} t^2.\]
Using \eqref{S} and the relations above, we see that for all $\eps>0$ small enough and all $\tau\ge 0$
\begin{align}\label{jh}
\tilde{J}(\tau w_{\varepsilon}) &\le \frac{\tau^2}{2}\|w_\eps\|^2-\frac{\tau^{2^*_s}}{2^*_s}\|w_\eps\|_{2^*_s}^{2^*_s}-\frac{C\tau^2}{2}\int_\Omega u_\mu^{2^*_s-2}w_\eps^2\,dx \\
\nonumber&\le \frac{\tau^2}{2} \Big[ \frac{1}{T(N,s)^2} + C {\varepsilon}^{N-2s}- C' {\varepsilon}^{2s}\Big] - \frac{\tau^{2_s^*}}{2_s^*} =: h_\eps(\tau)
\end{align}
($C,C' >0$ independent of $\varepsilon$). Now we focus on the mapping $h_{\varepsilon} \in C^1(\R_+)$. First we note that
\[\lim_{\tau\to \infty} h_{\varepsilon}(\tau)= - \infty,\]
so there exists $\tau_\eps\ge 0$ s.t.\
\[h_\eps(\tau_\eps) = \max_{\tau\ge 0}h_\eps(\tau).\]
If $\tau_\eps=0$, from \eqref{jh} we immediately deduce \eqref{Max}. So, let $\tau_\eps>0$. Differentiating $h_\eps$, we get
\[\tau_\eps = \Big[ \frac{1}{T(N,s)^2} + C {\varepsilon}^{N-2s}- C' {\varepsilon}^{2s}\Big]^{\frac{1}{2_s^* - 2}},\]
which tends to $T(N,s)^{-\frac{2}{2^*_s-2}}>0$ as $\eps\to 0^+$. So, taking $\eps>0$ small enough, we have $\tau_\eps\ge\tau_0>0$. Set
\[\tilde\tau_\eps = \Big[ \frac{1}{T(N,s)^2} + C {\varepsilon}^{N-2s}\Big]^{\frac{1}{2_s^* - 2}},\]
and note that the mapping
\[\tau\mapsto\frac{\tau^2}{2}\Big[ \frac{1}{T(N,s)^2} + C {\varepsilon}^{N-2s}\Big] - \frac{\tau^{2_s^*}}{2_s^*}\]
is increasing in $[0,\tilde\tau_\eps]$. So we have
\begin{align*}
h_\eps(\tau_\eps) &= \frac{\tau_\eps^2}{2}\Big[ \frac{1}{T(N,s)^2} + C {\varepsilon}^{N-2s}\Big]-\frac{\tau_\eps^{2^*_s}}{2^*_s}-\frac{C'\eps^{2s}\tau_\eps^2}{2} \\
&\le \frac{\tilde\tau_\eps^2}{2}\Big[ \frac{1}{T(N,s)^2} + C {\varepsilon}^{N-2s}\Big]-\frac{\tilde\tau_\eps^{2^*_s}}{2^*_s}-C''\eps^{2s} \\
&= \frac{s}{N}\Big[ \frac{1}{T(N,s)^2} + C {\varepsilon}^{N-2s}\Big]^\frac{N}{2s}-C''\eps^{2s}.
\end{align*}
Since $N-2s>2s$, for all $\eps>0$ small enough we have by \eqref{C}
\[h_\eps(\tau_\eps) < \frac{s}{NT(N,s)^\frac{N}{s}} = c^*.\]
Then, by \eqref{jh} we obtain \eqref{Max}. The cases $2s<N\le 4s$ are treated in similar ways, see \cite[Lemma 2.11]{BCSS}.
\vskip2pt
\noindent
As a byproduct of \eqref{jh} we have that $\tilde J(\tau w_\eps)\to -\infty$ as $\tau\to\infty$, so we can find $\bar\tau>0$ s.t.\
\[\tilde J(\bar\tau w_\eps) < 0.\]
Since $\tilde J$ has a local minimum at $0$ and no other critical point, we can find $\sigma\in(0,\|\bar\tau w_\eps\|)$ s.t.\ $\tilde J(v)>0$ for all $v\in\h$, $\|v\|=\sigma$. That is, $\tilde J$ exhibits a mountain pass geometry around $0$. Set
\[\Gamma=\big\{\gamma \in C([0,1], \h): \gamma(0)=0, \gamma(1)=\bar\tau w_{\varepsilon}\big\}, \quad  c = \inf_{\gamma \in \Gamma} \max_{t \in [0,1]} \tilde{J}(\gamma(t)).\]
Clearly, $\gamma(t)=t\bar\tau w_\eps$ define a path of the family $\Gamma$, so by \eqref{Max} we have
\[c \le \max_{t\in[0,1]}\tilde J(t\bar\tau w_\eps) < c^*.\]
Thus, $\tilde J$ satisfies $(PS)_c$. By the Mountain Pass Theorem, there exists $v\in\h\setminus\{0\}$ s.t.\ $\tilde J'(v)=0$ in $H^{-s}(\Omega)$, a contradiction.
\vskip2pt
\noindent
So we have proved the existence of $v_\mu\in\h\setminus\{0\}$ s.t.\ $\tilde J'(v_\mu)=0$ in $H^{-s}(\Omega)$. Such $v_\mu$ solves \eqref{Cr2}, and by monotonicity of $f$ we have for a.e.\ $x\in\Omega$
\[\tilde f(x,v_\mu(x)) = f(u_\mu(x)+v_\mu^+(x))-f(u_\mu(x)) \ge 0,\]
so by the fractional Hopf lemma (see for instance \cite[Lemma 2.7]{IMS}, as Proposition \ref{reg} here does not apply) we have $v_\mu\in{\rm int}(C^0_s(\overline\Omega)_+)$. Now set
\[w_{\mu} = u_{\mu} + v_{\mu} \in \mathrm{int}(C_s^0(\overline{\Omega})_+).\]
Clearly $w_\mu>u_\mu$ in $\Omega$, and for all $\varphi\in\h$ we have
\begin{align*}
\langle J'(w_{\mu}), \varphi \rangle &= \langle u_{\mu} + v_{\mu}, \varphi \rangle - \int_{\Omega} f(u_{\mu} + v_{\mu}) \varphi\,dx\\
&=\Big[\langle u_{\mu}, \varphi \rangle - \int_{\Omega} f(u_{\mu}) \varphi\,dx \Big] + \Big[\langle v_{\mu}, \varphi \rangle - \int_{\Omega} \tilde{f}(x,v_{\mu}) \varphi\,dx \Big]\\
&=\langle J'(u_{\mu}), \varphi \rangle + \langle \tilde{J}'(v_{\mu}), \varphi \rangle=0,
\end{align*}
so $w_\mu$ solves \eqref{bn}, which concludes the proof.
\end{proof}

\vskip4pt
\noindent
{\bf Acknowledgement.} Both authors are members of GNAMPA (Gruppo Nazionale per l'Analisi Matematica, la Probabilit\`a e le loro Applicazioni) of INdAM (Istituto Nazionale di Alta Matematica 'Francesco Severi') and supported by the research project {\em Integro-differential Equations and nonlocal Problems} funded by Fondazione di Sardegna (2017). A.\ Iannizzotto is also supported by the grant PRIN n.\ 2017AYM8XW: {\em Non-linear Differential Problems via Variational, Topological and Set-valued Methods}. We thank Prof.\ G.\ Bonanno for his precious suggestions.

\end{document}